\documentclass[a4paper,11pt]{article}
\usepackage{amsfonts,amssymb,amsthm,amsmath, dsfont,relsize,accents}
\usepackage{bm}
\usepackage{mathrsfs}  
\usepackage[usenames]{color}
\usepackage{upgreek}
\usepackage[english]{babel}
\usepackage{graphicx}
\usepackage{microtype}
\usepackage[colorlinks=true, pdfstartview=FitV, linkcolor=blue, citecolor=blue, urlcolor=blue,pagebackref=false]{hyperref}
\usepackage{graphicx}
\usepackage{epstopdf}
\definecolor{labelkey}{gray}{.8}
\definecolor{refkey}{gray}{.8}

\definecolor{darkred}{rgb}{0.9,0.1,0.1}

 \makeatletter
 \@addtoreset{equation}{section}
 \makeatother

 \makeatletter
 \@addtoreset{enunciato}{section}
 \makeatother

 \newcounter{enunciato}[section]

 \newtheorem{ittheorem}{Theorem}
 \newtheorem{itlemma}{Lemma}
 \newtheorem{itproposition}{Proposition}
 \newtheorem{itcorollary}{Corollary}
 \newtheorem{itdefinition}{Definition}
 \newtheorem{itremark}{Remark}
 \newtheorem{itclaim}{Claim}
 \newtheorem{itfact}{Fact}
 \newtheorem{itconjecture}{Conjecture}

 \newenvironment{theorem}{\addtocounter{enunciato}{1}
 \begin{ittheorem}}{\end{ittheorem}}

 \newenvironment{lemma}{\addtocounter{enunciato}{1}
 \begin{itlemma}}{\end{itlemma}}

 \newenvironment{proposition}{\addtocounter{enunciato}{1}
 \begin{itproposition}}{\end{itproposition}}

 \newenvironment{corollary}{\addtocounter{enunciato}{1}
 \begin{itcorollary}}{\end{itcorollary}}

 \newenvironment{definition}{\addtocounter{enunciato}{1}
 \begin{itdefinition}}{\end{itdefinition}}

 \newenvironment{remark}{\addtocounter{enunciato}{1}
 \begin{itremark}}{\end{itremark}}

 \newenvironment{claim}{\addtocounter{enunciato}{1}
 \begin{itclaim}}{\end{itclaim}}

 \newenvironment{fact}{\addtocounter{enunciato}{1}
 \begin{itfact}}{\end{itfact}}

 \newenvironment{conjecture}{\addtocounter{enunciato}{1}
 \begin{itconjecture}}{\end{itconjecture}}

 \newcommand{\be}[1]{\begin{equation}\label{#1}}
 \newcommand{\ee}{\end{equation}}

 \newcommand{\bl}[1]{\begin{lemma}\label{#1}}
 \newcommand{\el}{\end{lemma}}

 \newcommand{\br}[1]{\begin{remark}\label{#1}}
 \newcommand{\er}{\end{remark}}

 \newcommand{\bt}[1]{\begin{theorem}\label{#1}}
 \newcommand{\et}{\end{theorem}}

 \newcommand{\bd}[1]{\begin{definition}\label{#1}}
 \newcommand{\ed}{\end{definition}}

 \newcommand{\bcl}[1]{\begin{claim}\label{#1}}
 \newcommand{\ecl}{\end{claim}}

 \newcommand{\bfact}[1]{\begin{fact}\label{#1}}
 \newcommand{\efact}{\end{fact}}

 \newcommand{\bp}[1]{\begin{proposition}\label{#1}}
 \newcommand{\ep}{\end{proposition}}

 \newcommand{\bc}[1]{\begin{corollary}\label{#1}}
 \newcommand{\ec}{\end{corollary}}

 \newcommand{\bcj}[1]{\begin{conjecture}\label{#1}}
 \newcommand{\ecj}{\end{conjecture}}

 \newcommand{\bpr}{\begin{proof}}
 \newcommand{\epr}{\end{proof}}

 \newcommand{\bprlem}[1]{\begin{proofof}{\it Lemma \ref{#1}}.\,\,}
 \newcommand{\eprlem}{\end{proofof}}

 \newcommand{\bprthm}[1]{\begin{proofof}{\it Theorem \ref{#1}}.\,\,}
 \newcommand{\eprthm}{\end{proofof}}

 \newcommand{\bprprop}[1]{\begin{proofof}{\it Proposition \ref{#1}}.\,\,}
 \newcommand{\eprprop}{\end{proofof}}

 \newcommand{\bi}{\begin{itemize}}
 \newcommand{\ei}{\end{itemize}}

 \newcommand{\ben}{\begin{enumerate}}
 \newcommand{\een}{\end{enumerate}}

 \newenvironment{proofof}{\noindent {\em Proof of\,\,}}{\hspace*{\fill}$\halmos$\medskip}
 \newcommand{\halmos}{\rule{1ex}{1.4ex}}

 \newcommand{\one}{{\mathchoice {1\mskip-4mu\mathrm l}
         {1\mskip-4mu\mathrm l}
         {1\mskip-4.5mu\mathrm l}
         {1\mskip-5mu\mathrm l}}}

\def \E {{\mathbb E}}

\def \N {{\mathbb N}}
\def \P {{\mathbb P}}

\def \R {{\mathbb R}}
\def \Z {{\mathbb Z}}

\def \lra \leftrightarrow

\def \ra {\rightarrow}
\def \ba {\begin{array}}
\def \ea {\end{array}}

\def \lra {\longrightarrow}

\def \lra {{\leftrightarrow}}

\def \subset {\subseteq}

\def \emptyset {\varnothing}

\def\one{\rlap{\mbox{\small\rm 1}}\kern.15em 1}

\newlength{\dhatheight}

 \def \cmmvy {{\bar{c}}}
\def \csv {{\hat{c}}}
\def \bfgn {{\bf G}_n}   

\begin{document}
\title{Exponential rate for the contact process extinction time}
\author{Bruno Schapira\textsuperscript{1}, Daniel Valesin\textsuperscript{2}}
\footnotetext[1]{Aix-Marseille Universit\'e, CNRS, Centrale Marseille, I2M, UMR 7373, 13453 Marseille, France.\\ \url{bruno.schapira@univ-amu.fr}}
\footnotetext[2]{\noindent University of Groningen, Nijenborgh 9, 9747 AG Groningen, The Netherlands.\\ \url{d.rodrigues.valesin@rug.nl}}

\maketitle
\begin{abstract}
We consider the extinction time of the contact process on increasing sequences of finite graphs obtained from a variety of random graph models. Under the assumption that the infection rate is above the critical value for the process on the integer line, in each case we prove that the logarithm of the extinction time divided by the size of the graph converges in probability to a (model-dependent) positive constant. The graphs we treat include various percolation models on increasing boxes of~$\mathbb{Z}^d$ or~$\mathbb{R}^d$ in their supercritical or percolative regimes (Bernoulli  bond and site percolation, the occupied and vacant sets of random interlacements, excursion sets of the Gaussian free field, random geometric graphs) as well as supercritical Galton-Watson trees grown up to finite generations.
\end{abstract}

\section{Introduction}
In this paper, continuing our earlier work \cite{SV}, we present a robust method allowing to prove the existence of an exponential rate of convergence of the contact process extinction time on various models of random graphs, when the infection parameter is large enough. 

In the contact process, each vertex of a graph is at any point in time either healthy (state~0) or infected (state~1). The continuous-time dynamics is defined by the specification that infected vertices become healthy with rate one, and transmit the infection to each neighboring vertex with rate $\lambda > 0$. We refer to \cite{Lig} for a standard reference on the contact process.

The configuration in which all vertices are healthy is absorbing for the dynamics, and in finite graphs it is reached with probability one. In certain situations, the dynamics stays active for a very long time before reaching this absorbing state, and this behavior can be understood as an instance of metastability. To be more precise, let~$\uptau_G$ be the hitting time of the absorbing, ``all-healthy'' (or ``empty'') state when the process is started from the configuration in which all vertices are infected. The distribution of~$\uptau_G$ depends on the graph~$G$ and the infection rate~$\lambda$. One typically considers a sequence of graphs~$(G_n)$ (which could be deterministic and nested, or an increasing sequence of random graphs from the same random graph model), fixes~$\lambda > 0$ and studies the asymptotic behavior of~$\uptau_{G_n}$. For a variety of sequences~$(G_n)$, it is known that, if~$\lambda$ is large enough, then there exists $c > 0$ such that
\begin{equation}\label{eq:at_least_exp}
\mathbb{P}\left(\uptau_{G_n} > \exp\{c|G_n|\}\right)\xrightarrow{n \to \infty} 1,
\end{equation}
where for any graph $G$ we let $|G|$ denote its number of vertices (see for instance Theorem~3.9 in Section~I.3 of~\cite{Lig}, Theorem~1.3 in~\cite{MMVY} and Theorem~1.5 in~\cite{CMMV13}). Evidently, the meaning of ``$\lambda$ large enough'' depends on the chosen sequence of graphs. For instance, for certain random graphs with power law degree distributions, $\lambda > 0$ suffices, whereas for boxes of $\mathbb{Z}^d$, one must take $\lambda$ larger than $\lambda_c(\mathbb{Z}^d)$, the critical value for the $d$-dimensional contact process, defined as the supremum of values of~$\lambda$ for which the process on~$\mathbb{Z}^d$ started from a single infection almost surely reaches the empty configuration.
Recently, in \cite{MMVY} and \cite{SV}, it has been established that for values of the infection rate above $\lambda_c(\mathbb{Z})$ -- the critical value of the one-dimensional contact process --, statement~\eqref{eq:at_least_exp} (or at least a slightly weaker result in which~$\exp\{c|G_n|\}$ is replaced by~$\exp\{c|G_n|/(\log |G_n|)^{1+\varepsilon}\}$) holds for \textit{arbitrary} sequences of connected graphs~$(G_n)$ with~$|G_n| \to \infty$.

A natural refinement of~\eqref{eq:at_least_exp} is the statement that $\frac{1}{|G_n|}\cdot \log \uptau_{G_n}$ converges (in some sense) to something positive as~$n \to \infty$. The first result of this kind was given in~\cite{DS88}: there it is proved that for~$\lambda > \lambda_c(\mathbb{Z})$ and $(G_n)$ given by line segments of length~$n$, there exists a constant~$\gamma \in (0,\infty)$ such that
\begin{equation}\label{eq:conv_in_prob}
\frac{1}{|G_n|}\cdot \log(\uptau_{G_n}) \xrightarrow{n \to \infty} \gamma \text{ in probability}.
\end{equation}
This was generalized in~\cite{mo99} to boxes of~$\mathbb{Z}^d$ (with~$\lambda > \lambda_c(\mathbb{Z}^d)$), and the same result was proved in~\cite{CMMV13} for~$d$-regular trees truncated at height~$n$ (with~$\lambda$ larger than the \textit{upper} critical value of the contact process on the infinite~$d$-regular tree; see~\cite{Lig} for the definition).

Our goal in this paper is to establish results of the type~\eqref{eq:conv_in_prob} for sequences~$({\bf G}_n)$ obtained from random graph models. As we rely on techniques developed in our previous work~\cite{SV}, where the assumption~$\lambda > \lambda_c(\mathbb{Z})$ was crucial, we also need this assumption here, though we do not believe it to be sharp in any of the settings we consider. This being said, our method is quite robust, and depends essentially on the graphs ${\bf G}_n$ possessing some kind of self-similar structure which allows for a recursive decomposition. 

Let us list the choices of sequences of random graphs~$({\bf G}_n)$ covered by our main theorem. For now we only refer to these random graph models by their names and the assumptions we place on their defining parameters; in Sections~\ref{s:percolation} and~\ref{s:Galton-Watson}, we will present each model in detail. We abuse notation and denote by~$B_n$ both the set~$\{-n,\ldots,n-1\}^d \subset \mathbb{Z}^d$ and the set~$[-n,n]^d \subset \mathbb{R}^d$.
\begin{enumerate}
\item For $d \geq 2$, perform supercritical Bernoulli bond percolation on~$B_n$ and let~${\bf G}_n$ be the resulting maximal component (that is, the connected component with largest cardinality).
\item For~$d \geq 3$, perform supercritical Bernoulli site percolation on $B_n$ and let~${\bf G}_n$ be the resulting maximal component. 
\item For~$d \geq 3$ and~$u > 0$, let~$\mathcal{I}_u$ be the occupied set of random interlacement with intensity~$u$ on~$\mathbb{Z}^d$; let~${\bf G}_n$ be the maximal component of the subgraph of~$\mathbb{Z}^d$ induced by~$\mathcal{I}_u \cap B_n$.
\item For~$d \geq 3$ and~$u > 0$ sufficiently small, let~$\mathcal{V}_u$ be the vacant set of random interlacement with intensity~$u$ on~$\mathbb{Z}^d$; let~${\bf G}_n$ be the maximal component of the subgraph of~$\mathbb{Z}^d$ induced by~$\mathcal{V}_u \cap B_n$.
\item For~$d \geq 3$, let~$\varphi = (\varphi_x)_{x \in \mathbb{Z}^d}$ be the Gaussian free field on~$\mathbb{Z}^d$. Let~$E_\varphi^{\geq h} = \{x \in \mathbb{Z}^d:\;\varphi_x \geq h\}$ for~$h \in \mathbb{R}$ small enough, and let~${\bf G}_n$ be the maximal component of the subgraph of~$\mathbb{Z}^d$ induced by~$E_\varphi^{\geq h} \cap B_n$.
\item For~$d \geq 2$, consider the supercritical random geometric graph on~$B_n$ and let~${\bf G}_n$ be the resulting maximal component.
\item Let~$\nu$ be a probability measure on~$\mathbb{N}$ with $\sum_k k\nu(k) > 1$ and $\sum_k k^2 \nu(k) < \infty$. Let~${\bf G}$ be a Galton-Watson tree with offspring distribution~$\nu$ and either conditioned on being infinite, or conditioned on reaching generation $n$. Let~${\bf G}_n$ be the subgraph of~${\bf G}$ induced by the set of vertices at distance at most~$n$ from the root. 
\end{enumerate}
\begin{theorem}\label{thm:thm}
For any~$\lambda > \lambda_c(\mathbb{Z})$ and any of the choices of $({\bf G}_n)$ listed above, there exists a (deterministic) constant~$\gamma \in (0,\infty)$ such that the extinction time of the contact process on ${\bf G}_n$ with rate~$\lambda$ and started from full occupancy satisfies:
$$\frac{1}{|{\bf G}_n|} \cdot \log \uptau_{{\bf G}_n} \xrightarrow{n \to \infty} \gamma \text{ in probability.} $$
\end{theorem}

Our main tool to prove Theorem~\ref{thm:thm} is Proposition~\ref{supmult} below, which was proved in our previous paper \cite{SV}; it states that on any (deterministic) graph, and for any  
$\lambda$ larger than $\lambda_c(\Z)$, the expected extinction time is larger than the product of the expected extinction times on any collection of 
disjoint connected subgraphs, up to some correction term. This result allows us to use a block decomposition in the same vein as in the proof of Mountford 
\cite{mo99} in the case of boxes of~$\mathbb{Z}^d$. The main graph property which is required is as already mentioned some kind of self-similarity, namely that at each scale $n$, 
the graph~${\bf G}_n$ can be decomposed as a union of subgraphs, which are copies (in law) of the original graph at a smaller scale.

Few results are available for the contact process on random graphs obtained from percolation-type models; see for instance~\cite{Ber, ChenYao, Xue} and references therein. The contact process on the supercritical random geometric graph has also been previously considered by M\'enard and Singh \cite{MS}, who proved that 
the critical infection rate is positive, and by Can \cite{Can} who obtained sharp bounds on the expected value of the 
extinction time on ${\bf G}_n$ when the radius of connectedness goes to infinity. The contact process on Galton-Watson trees has been studied by in Pemantle~\cite{Pem}. 

Finally, let us mention that for several important random graph models, it would be interesting to obtain results of the form~\eqref{eq:conv_in_prob}, but our present techniques are not applicable (at least not directly). These include the configuration model, the Erd\' os-Renyi random graph, random planar maps, and Delaunay triangulations of the plane (provided that in each case, the parameters defining the graph and the value of~$\lambda$ yield a regime of exponentially large extinction time).

The paper is organized as follows. In the next section we gather known important results for the contact process. In Section 3, we consider all the percolation-type models we treat (models 1 to 6 in the list above), 
starting with those which are subgraphs of $\Z^d$. In this case, very similarly to \cite{DRS} and \cite{Sap}, we list a number of general conditions under which the conclusion of Theorem \ref{thm:thm} hold true, and which are known to be satisfied in all the models mentioned above. The case of Galton-Watson trees is treated separately in the last section, as the proof in this setting presents some substantial differences.  

\subsection{Notation and conventions for graphs}
We use the convention that a graph $G=(V,E)$ is a collection of vertices $V$ and edges $E$ between vertices, and we assume that graphs are undirected, with no loops, and no multiple edges. They are also assumed to be locally finite, meaning that any vertex has only finitely many neighbors (or edges emanating from this vertex).

We denote by $|V|$ and $|E|$ respectively the number of vertices and edges in the graph, and use the convention that $|G|$ also denotes the number of vertices in the graph. 

If~$G = (V,E)$ and~$V' \subseteq V$, the subgraph of~$G$ induced by~$V'$ is the graph~$G' = (V',E')$, where~$E'$ is the set of edges of~$E$ with both extremities in~$V'$.

Finally, we use the notation~$\Z^d$ to refer both to the set of~$d$-dimensional vectors with integer coordinates, and to the graph with these vectors are vertices, and edges connecting vectors at~$\ell_1$-distance one from each other.

\section{Prelimaries on the contact process on finite graphs}

For a full account of the contact process on~$\Z^d$ and other locally finite graphs, we refer the reader to the book of Liggett \cite{Lig}. Here we merely recall some bounds on the expected extinction time of the process on finite graphs. We start with the following basic fact, which provides a general upper bound, and indicate its short proof for the reader's convenience.  
\begin{lemma}\label{lem:unif_bound}
For any $\lambda>0$, and any finite connected graph $G=(V,E)$, 
$$\mathbb E[\uptau_G] \le  \exp(|V| + 2\lambda |E|).$$
\end{lemma}
\begin{proof}
The probability that extinction occurs before time $1$ is larger than the probability that all vertices recover and no transmission occurs before time $1$, which is  at least~$\exp(-|V|-2\lambda |E|)$. Then,~$\uptau_G$ is stochastically dominated by a random variable following the geometric distribution with parameter~$\exp(-|V|-2\lambda |E|)$; the expectation of this random variable is~$\exp(|V| + 2\lambda |E|).$
\end{proof}

On the other hand, general lower bounds were provided in \cite{MMVY} and \cite{SV}: 
\begin{theorem} \label{thm:lower}
Assume $\lambda > \lambda_c(\mathbb{Z})$.
\begin{itemize}
\item[\cite{MMVY}] For any $d > 0$, there exists $\cmmvy > 0$ such that for any connected graph $G$ with degrees bounded by $d$, we have 
$$\mathbb{E}[\uptau_G] > \exp\{\cmmvy |G|\}.$$
\item[\cite{SV}] For any $\varepsilon > 0$, there exists $\csv > 0$ such that for any connected graph $G$ with at least two vertices, we have
$$\mathbb{E}[\uptau_G] > \exp\left\{\frac{\csv|G|}{(\log|G|)^{1+\varepsilon}} \right\}. $$ 
\end{itemize}
\end{theorem}
Next, let us recall the general result from \cite{SV} alluded to in the introduction, which is also our main tool. 
\begin{proposition}[\cite{SV}, Proposition 2.9] 
\label{supmult}
There exists a constant $c_0>0$, such that for any finite connected graph $G$ and any $\lambda>\lambda_c(\Z)$, the following holds: 
for any $N\ge 1$, and any finite collection of disjoint connected subgraphs $G_1,\dots,G_N\subset G$, one has 
$$\E[\uptau_G] \ge \frac{c_0}{(2|G|^{3})^{N+1}} \cdot \prod_{i=1}^{N} \E[\uptau_{G_i}].$$
\end{proposition}
We note that this result was stated and proved for trees only in \cite{SV}, but the identical statement, with the same proof, works for general graphs. 

Finally, the following is useful in turning bounds on expectations into bounds on probabilities:
\begin{lemma}[\cite{MMVY}, Lemma 4.5]\label{lem:inverse_cheb}
For any $t > 0$,
\begin{equation*}
\mathbb{P}\left(\uptau_G \leq t\right) \leq \frac{t}{\mathbb{E}[\uptau_G]}.
\end{equation*}
\end{lemma}

\section{Percolation-type models}\label{s:percolation}
\subsection{General lattice model}\label{ss:general}

We consider here random graphs obtained from percolation-type models on $\Z^d$ (models 1-5 in the list before Theorem~\ref{thm:thm}). Rather than treating model by model, we will be able to treat them all at once by following an approach similar of that of~\cite{DRS} and~\cite{Sap}. That is, we will first present a list of four properties to be satisfied by a probability measure~$\mathbb{P}$ under which a random subgraph~${\bf G}$ of~$\Z^d$ is defined. Next, letting~${\bf G}_n$ be the maximal component of~${\bf G} \cap B_n$, we will prove the statement of Theorem~\ref{thm:thm} by only making use of the mentioned properties. Finally, at the end of this subsection, we will give formal definitions of the aforementioned five models (Bernoulli bond and site percolation, occupied and vacant set of random interlacements, excursion set of the Gaussian free field), and provide references that show that each model satisfies the four properties (at least for some choices of their defining parameters).

We let~$\Omega = \{0,1\}^S$, where~$S$ is either the set of vertices or edges of~$\Z^d$ (as usual, elements of~$\Omega$ can be identified with subsets of the vertex (or edge) set of~$\mathbb{Z}^d$). In the percolation terminology, given~$\omega \in \Omega$, a vertex (or edge) is said to be open if its~$\omega$-value is one, and to be closed otherwise. A configuration~$\omega \in \Omega$ then defines a subgraph~${\bf G} = {\bf G}(\omega)$ of~$\Z^d$ in the standard manner. That is, if~$S$ is the set of vertices of~$\Z^d$, then~${\bf G}$ is the subgraph of~$\Z^d$ induced by~$\{x: \omega_x = 1\}$, and if~$S$ is a set of edges, then~${\bf G}$ has all the vertices of~${\Z^d}$ and the set of edges~$\{e: \omega_e = 1\}$. (The second situation is only needed to include the case of Bernoulli bond percolation). We endow~$\Omega$ with the~$\sigma$-algebra generated by the coordinate maps, and note that any probability measure~$\P$ on~$\Omega$ then defines a random subgraph ${\bf G}$ of~$\Z^d$.

Our two first assumptions on $\P$ are the following: 
\begin{equation}\label{H0}
\text{$\P$ is translation invariant and ergodic with respect to spatial shifts of~$\mathbb{Z}^d$}
\end{equation}
and 
\begin{equation}\label{H0'}
\text{$\P$-almost surely $\bf G$ has a unique infinite connected component, denoted by ${\bf G}_\infty$. }
\end{equation}
Under these hypotheses, for any $n\ge 1$,   
we let ${\bf G}_n$ be the maximal connected component of~${\bf G}$ intersected with the box $B_n:=\{-n,\dots,n-1\}^d$ (with the convention that 
if there are more than one maximal connected component, 
one chooses ${\bf G}_n$ among them in an arbitrary way).

Our next assumption ensures that with high probability, connected components of~${\bf G} \cap B_n$ distinct from~${\bf G}_n$ have small diameter. So under \eqref{H0'}, we consider 
\begin{align}\label{H2}
&\nonumber  \hspace{2cm} \text{there exists a constant $\Delta>0$, such that for all $n$ large enough,} \\
&\P\left(\begin{array}{c}
{\bf G}_\infty \cap B_n \neq \emptyset, \text{ and any two} \\
\text{connected components of ${\bf G}\cap B_n$ with diameter}\\
\text{larger than $\frac n{10}$ are connected in $B_{2n}$}
\end{array}\right)\ge 1- \exp(-(\log n)^{1+\Delta}).
\end{align}
Note that when the above assumption holds, one can easily derive the following useful fact: for all $n$ large enough,
\begin{align}\label{H2'}
\P\left(\begin{array}{c}
\text{any two connected components of ${\bf G}_\infty \cap B_n$} \\
\text{are connected in $B_{4n}$}
\end{array}\right)\ge 1- \exp(-(\log 2n)^{1+\Delta}).
\end{align}
Indeed, any connected component of ${\bf G}_\infty\cap B_{2n}$ intersecting $B_n$ necessarily has a diameter larger than $n$, so that one can apply \eqref{H2} in $B_{2n}$. 
Let us also mention that for our purposes this hypothesis \eqref{H2} could be weakened, for instance a super-polynomial decay as in Section 5 of \cite{Sz2} would be sufficient.

Finally our last assumption gives a bound on the decay of correlations for events depending on the configurations inside two disjoint boxes, as the distance between the two boxes diverges. Note that we have not tried to 
find an optimal condition for the proof here, and so we assume that the decay of correlations is at least polynomially fast. Indeed this is good enough for the proof, and it is satisfied by all the models of interest 
to us here.  
We let $\|\cdot\|$ be the Euclidean norm. Then we consider 
\begin{align}\label{H3}
\begin{array}{c} 
\text{there exists a constant $C_0$ such that for any $x\in \Z^d$, and any events $A$ and $B$}\\
\text{depending only on the configurations inside $B_n$ and $x+B_n$ respectively, one has:} \\
|\text{Cov}({\mathds 1}_A,{\mathds 1}_B)| \le C_0 (\frac{n^2}{\|x\|})^{d-2}. 
\end{array}
\end{align}
We restate now Theorem \ref{thm:thm} in this general setting.

\begin{theorem}\label{thm:perco}
Let $\P$ denote a probability measure on $\{0,1\}^{\Z^d}$, $d\ge 2$, satisfying \eqref{H0}, \eqref{H0'}, \eqref{H2} and \eqref{H3}, and let 
$({\bf G}_n)_{n\ge 1}$ and ${\bf G}_\infty$ be as defined above. Consider now the contact process on ${\bf G}_n$, with infection rate $\lambda>\lambda_c(\Z)$, starting from full occupancy, and denote by $\uptau_{{\bf G}_n}$ its extinction time. There exists $\gamma \in (0,\infty)$ such that 
\begin{equation}\label{conv.tau.perco}
\frac 1{|{\bf G}_n|}\cdot  \log \uptau_{{\bf G}_n} \quad \xrightarrow[n\to \infty]{(\P)} \quad \gamma.
\end{equation}
\end{theorem}

Before proving this result let us start with some preliminary facts concerning the sizes of the connected components inside $B_n$.   
\begin{proposition}\label{prop:component}
Let $\P$ denote a probability measure on $\{0,1\}^{\Z^d}$, $d\ge 2$, satisfying~\eqref{H0},~\eqref{H0'}, and~\eqref{H2}. Then with the above notation, for any $\varepsilon \in (0,1)$ there exists $\delta > 0$ such that for $n$ large enough the following holds with probability higher than $1-\exp\{-(\log n)^{1+\delta}\}$:
\begin{align}\label{eq:prop_comp1} &\text{${\bf G}_\infty \cap B_n$ has a single component of cardinality larger than~$n^{d-\varepsilon}$};\\[.2cm]
\nonumber&\text{all other components of ${\bf G} \cap B_n$ either have cardinality smaller}\\[-.5cm] \label{eq:prop_comp2}\\&\hspace{4cm} \text{than $n^\varepsilon$ or are contained in $B_n \backslash B_{ n-n^\varepsilon }$.}\nonumber
\end{align}
\end{proposition}
\begin{proof}
Fix~$\varepsilon > 0$, and define~$\ell_n = \lfloor n^{\varepsilon/d}/2\rfloor$. Given a connected graph $G \subset \mathbb{Z}^d$, we say $G$ crosses an annulus~$x + (B_{\ell_n} \backslash B_{\ell_n/4})$ if the vertex set of~$G$ intersects both~$x + B_{\lfloor \ell_n/4\rfloor}$ and~$x+ (B_{\ell_n})^c$.

By \eqref{H2}, \eqref{H2'} and a union bound, there exists~$\delta > 0$ such that, for~$n$ large enough, the following conditions are  satisfied with probability at least~$1-\exp\{-(\log n)^{1+\delta}\}$:
\begin{itemize}
\item for any $x \in B_n$, the graph ${\bf G}_\infty\cap (x + B_{\ell_{2n}})$ has a unique component $\mathscr{C}(x)$ that crosses~$x + (B_{\ell_n} \backslash B_{\ell_n/4})$;
\item for any $x\in B_n$, any connected component of ${\bf G}\cap (x+B_{\ell_n})$ with diameter larger than $\ell_n/10$ is connected to $\mathcal C(x)$ in $x+B_{2\ell_n}$;
\item for any $x, y \in B_{n}$ with $x \sim y$, we have that $\mathscr{C}(x)$ and $\mathscr{C}(y)$ are connected together in ${\bf G}_\infty \cap (x + B_{4\ell_n})$.
\end{itemize}
In particular, all the graphs $\mathscr{C}(x)$ for $x \in B_{n-8\ell_n}$ are subgraphs of a single component $\mathscr{C}^\star$ of~${\bf G}_\infty \cap B_n$. Noting that $\mathscr{C}^\star$ intersects all boxes~$x+B_{\ell_n}$ for $x \in B_{n - 8\ell_n}$, it follows that~$\mathscr{C}^\star$ has at least~$\lfloor (n-8\ell_n)^d/(\ell_n)^d \rfloor \geq n^{d-\varepsilon}$ vertices.

Now, let~$\mathscr{C}$ be some component of~${\bf G}\cap B_n$ distinct from~$\mathscr{C}^\star$. Then,~$\mathscr{C}$ cannot cross any annulus~$x + (B_{\ell_n} \backslash B_{\ell_n/4})$ with~$x \in B_{n-8\ell_n}$. It thus follows that at least one of the following two conditions hold:
\begin{itemize}\item $\mathscr{C}$ has diameter smaller than $\ell_n$ (so $|\mathscr{C}|\leq n^\varepsilon$);
\item $\mathscr{C}$ is contained in~$B_n \backslash B_{n-9\ell_n} \subset B_n \backslash B_{n - n^\varepsilon}$ (so $|\mathscr{C}|\leq 2d\cdot n^{d-1} \cdot 9\ell_n$).
\end{itemize}
This also shows that~$\mathscr{C}^\star$ is the unique component of~${\bf G} \cap B_n$ with cardinality above $n^{d-\varepsilon}$, completing the proof.
\end{proof}

As a corollary, one obtains the following result, interesting in itself. 

\begin{corollary}
Under the hypotheses of Proposition \ref{prop:component}, 
\begin{equation}\label{conv.size.perco}
\frac{|{\bf G}_n|}{|B_n|}\quad \xrightarrow[(\P)]{n\to \infty} \quad \theta:=\mathbb P(0\in {\bf G}_\infty).
\end{equation}
\end{corollary}
\begin{proof}
Fix $\varepsilon \in (0,1)$. Also fix $n \in \mathbb{N}$ and assume that \eqref{eq:prop_comp1} and \eqref{eq:prop_comp2} hold. Then,~${\bf G}_n$ is a component of ${\bf G}_\infty \cap B_n$.

If $\mathscr{C}$ is any  component of ${\bf G} \cap B_n$ distinct from ${\bf G}_n$, then we must have  either~$|\mathscr{C}| < n^\varepsilon$ or~$\mathscr{C} \subset B_n\backslash B_{n-n^\varepsilon}$. If we also assume~$\mathscr{C}$ is a component of~${\bf G}_\infty \cap B_n$, it must be the case that~$\mathscr{C}$ intersects the  boundary $B_n \backslash B_{n-1}$ (since ${\bf G}_\infty$ is connected), so we necessarily have~$\mathscr{C} \subset B_n \backslash B_{n-n^\varepsilon}$. Then,
$$
|{\bf G}_\infty \cap B_n| - |{\bf G}_n| = \sum_{\substack{\mathscr{C} \text{ component}\\\text{of } {\bf G}_\infty \cap B_n:\\ \mathscr{C} \neq {\bf G}_n}} |\mathscr{C}| \leq |B_n \backslash B_{n - n^\varepsilon}|  = o(n^d).$$
The desired result now follows from noting that, by ergodicity, $|{\bf G}_\infty \cap B_n|/|B_n|$ converges in probability to $\theta$ as~$n \to \infty$.
\end{proof}

We are now in position give the proof of Theorem \ref{thm:perco}. 

\begin{proof}[Proof of Theorem \ref{thm:perco}]
Define
\begin{equation} X_n:=\frac{1}{n^d} \log \E[\uptau_{{\bf G}_n}\mid {\bf G}_n] \label{eq:def_of_X}\end{equation}
and 
\begin{equation}\label{eq:def_of_gamma}
\widetilde \gamma := \limsup_{n \to \infty} \mathbb{E}[X_n].
\end{equation}
By Lemma~\ref{lem:unif_bound}, the $(X_n)_{n\ge 1}$ are bounded random variables, and thus $\widetilde \gamma$ is finite. 
We claim that it is also positive. Indeed by~Theorem~\ref{thm:lower} (noting that~$\bfgn$ is by definition connected and has degrees bounded by~$2d$), for $n$ large enough,
\begin{equation}\label{eq:proof_of_positivity}\begin{split}
\mathbb{E}[X_n] &\geq  \frac{1}{n^d}\cdot \mathbb{E}\left[\log \mathbb{E}[\uptau_{n}\mid {\bf G}_n] \cdot \mathds{1} \left\{ |\bfgn| > \frac{\theta}{2} \, |B_n|\right\} \right]
\\[2mm]&\geq  \frac{1}{n^d} \cdot \frac{\cmmvy \theta (2n)^d}{2}  \cdot \mathbb{P}\left(|\bfgn| > \frac{\theta}{2} \, |B_n|\right)\stackrel{\eqref{conv.size.perco}}{\ge}  \cmmvy  \theta 2^{d-2}.
\end{split}\end{equation} 
Now we prove that $(X_n)_{n\ge 0}$ converges in probability to $\widetilde \gamma$. Fix~$\varepsilon > 0$, and let~$n_0$ be such that~$\E[X_{n_0}]\ge \widetilde \gamma - \varepsilon/2$. In what follows we will also assume that $n_0$ is large enough (depending on~$\varepsilon$). 
We now write for $k \in \mathbb{N}$,
$$n_k = (n_0)^{3^k},\qquad \varepsilon_k = 2^{-k}\varepsilon,\qquad \mathcal E_k = \sum_{i=1}^k  \varepsilon_i.$$
We proceed by induction. Fix~$k \ge 0$, and assume that we have already proved that
\begin{equation}\label{eq:induction_assume}
\mathbb{E}[X_{n_k}] \ge \widetilde \gamma - \mathcal E_{k+1}.\end{equation} 
Then we show that for any~$n \in  [n_{k+1},n_{k+2}]$, we have
\begin{equation}\label{eq:main_bound_perc}\begin{split}
&\mathbb{P}\left(X_n \leq \widetilde \gamma - \mathcal E_{k+1} - \varepsilon_{k+3} \right) < \frac{ \varepsilon_{k+3}}{\widetilde \gamma},\\&\hspace{1cm}\text{ so } \mathbb{E}[X_n] > \left(\widetilde \gamma - \mathcal E_{k+1}- \varepsilon_{k+3}\right)\cdot \left(1-\frac{ \varepsilon_{k+3}}{\widetilde \gamma}\right) \geq \widetilde \gamma - \mathcal E_{k+2}.
\end{split}\end{equation}
To this end, fix $n \in [n_{k+1},n_{k+2}]$. We split the box~$B_{n-n_k}$ into $N:=(\lfloor n/n_k\rfloor - 1)^ d$ disjoint boxes of side length 
$2n_k$, and denote by 
$({\bf G}^{(i)})_{1\le i\le N}$ the maximal connected components inside each of these boxes. 
Define the event 
\begin{equation}\label{eq:def_of_A} A = \left\{\text{all the $({\bf G}^{(i)})_{1\le i\le N}$ are subgraphs of ${\bf G}_n$}\right\}.\end{equation} 
By Proposition~\ref{supmult}, on~$A$ we have 
$$\mathbb{E}[\uptau_{{\bf G}_n}\mid {\bf G}_n] \geq \frac{c_0}{(2|\bfgn|)^{3N}}\cdot \prod_{i=1}^N \mathbb{E}[\uptau_{{\bf G}^{(i)}}\mid {\bf G}^{(i)}], $$
so that, with $X^{(i)}:=(1/n_k^d)\log \mathbb{E}\left[\uptau_{{\bf G}^{(i)}} \mid {\bf G}^{(i)}\right]$, 
\begin{equation}\label{eq:interm}
X_n \cdot \mathds{1}_{ A}\geq \left(\left( \frac{n_k}{n}\right)^d \cdot \sum_{i=1}^N X^{(i)} - \frac{C N \log(|\bfgn|)}{n^d}\right) \cdot \mathds{1}_{A},
\end{equation}
for some constant $C>0$. 
Now note that, by the definition of $N$,
$$\frac{C N \log(|\bfgn|)}{n^d} \leq \frac{C dn^d \log(2n)}{n^d_k\cdot n^d} \leq   \varepsilon_{k+5}, $$
at least for $n_0$ large enough. Moreover, Lemma \ref{lem:unif_bound} implies that all the $X^{(i)}$ are bounded. Therefore, 
\begin{align*}
\left( \frac{1}{N}-\left( \frac{n_k}{n}\right)^d\right) \sum_{i=1}^N X^{(i)} \leq C_1 \left( 1-N\left( \frac{n_k}{n}\right)^d\right)  \leq  C_2\left(1 - \frac{n-2n_k}{n_k} \cdot\frac{n_k}{n} \right)  \leq  \varepsilon_{k+5},
\end{align*}
for some positive constants $C_1$ and $C_2$. 
Using these estimates in \eqref{eq:interm} we get
\begin{equation*}
X_n \cdot \mathds{1}_{A} \geq \left(\frac{1}{N} \cdot \sum_{i=1}^N X^{(i)} -  \varepsilon_{k+4}\right) \cdot \mathds{1}_{ A}. 
\end{equation*}
Consequently,
\begin{equation} 
\mathbb{P}(X_n \leq \widetilde \gamma - \mathcal E_{k+1} - \varepsilon_{k+3}) \leq \mathbb{P}( A^c) + 
\mathbb{P}\left(\frac{1}{N} \cdot \sum_{i=1}^N X^{(i)} \leq \widetilde \gamma - \mathcal E_{k+1} -  \varepsilon_{k+4} \right).\label{eq:want_to_bound}
\end{equation}
We now apply Proposition \ref{prop:component} in each of the boxes and deduce that for some $\delta>0$, with probability at least $1- \exp\{-(\log n)^{1+\delta}\}$, all the ${\bf G}^{(i)}$ have 
cardinality larger than $n_k^{d-1}$ and are subgraphs of ${\bf G}_\infty$. Note also that the ${\bf G}^{(i)}$ are by definition at distance at least $n_k$ from the boundary of $B_n$, since they all belong to $B_{n-n_k}$, and by hypothesis one has $n_k \ge n^{1/9}$. Thus by applying again Proposition \ref{prop:component} in the box $B_n$, we get 
that for $n_0$ large enough, and for some $\delta>0$, 
\begin{equation}\label{PAc}
\P(A^c)\ \le\ \exp\{-(\log n)^{1+\delta}\} \ \le\ \frac{\varepsilon_{k+4}}{\widetilde \gamma}.
\end{equation}
We now bound the second term in \eqref{eq:want_to_bound} using the hypothesis \eqref{H3}. 
To this end we discretize the random variables $(X^{(i)})_{i\le N}$, which we recall are bounded by a constant $C_1>0$. 
Set 
$$\widetilde X^{(i)} : =\varepsilon_{k+5} \sum_{\ell = 0}^{\lfloor C_1/\varepsilon_{k+5}\rfloor } \mathds 1\{X^{(i)}\ge \ell \varepsilon_{k+5}\}.$$
Note that 
$$\widetilde X^{(i)} \ \le X^{(i)} \ \le \widetilde X^{(i)} + \varepsilon_{k+5},$$
and that by \eqref{H0}, all the $(X^{(i)})_{i\le N}$ have the same distribution as $X_{n_k}$. 
Therefore \eqref{eq:induction_assume} gives, using also Chebyshev's inequality, 
\begin{align}\label{covariances}
\nonumber \mathbb{P}\left(\frac{1}{N} \cdot \sum_{i=1}^N X^{(i)} \leq \widetilde \gamma - \mathcal E_{k+1} -  \varepsilon_{k+4} \right)&  
\stackrel{\eqref{eq:induction_assume}}{\le}  \P\left( \frac 1N \sum_{i=1}^N \widetilde X^{(i)} -\E[\widetilde X^{(i)}]\le  - \varepsilon_{k+5}\right) \\
&  \le  \frac{1}{N^2\varepsilon_{k+5}^2} \cdot \sum_{1\le i,j\le N} |\text{Cov}(\widetilde X^{(i)},\widetilde X^{(j)})|.
\end{align}
We now bound the covariances using \eqref{H3}. 
Note first that by using the bilinearity of the covariance and \eqref{H3}, we get that for all $i,j$,  
$$|\text{Cov}(\widetilde X^{(i)},\widetilde X^{(j)})| \le C_0 C_1^2 \frac{n_k^{2(d-2)}}{d(i,j)^{d-2}},$$
where by $d(i,j)$ we mean the distance between the two boxes containing ${\bf G}^{(i)}$ and ${\bf G}^{(j)}$. We deduce that for any fixed $i\le N$, for some constant $C>0$, 
$$\sum_{1\le j\le N} |\text{Cov}(\widetilde X^{(i)},\widetilde X^{(j)})| \le C n_k^{d-2} \cdot (\frac n{n_k})^2.$$
Summing next over $i$, and using that $n\ge n_k^3$, we obtain  
$$\frac 1{N^2}\sum_{1\le i,j\le N} |\text{Cov}(\widetilde X^{(i)},\widetilde X^{(j)})| \le C \cdot  \frac{n_k^{d-2}}{N} \left(\frac{n}{n_k}\right)^2\le C\cdot \frac 1{n_k^{d-2}}.$$
Combining this with \eqref{covariances} gives, for $n_0$ large enough, 
$$\mathbb{P}\left(\frac{1}{N} \cdot \sum_{i=1}^N X^{(i)} \leq \widetilde \gamma - \mathcal E_{k+1} -  \varepsilon_{k+4} \right)\le \frac{\varepsilon_{k+4}}{\widetilde \gamma}.$$
Together with \eqref{eq:want_to_bound} and \eqref{PAc} this proves \eqref{eq:main_bound_perc}.

One can then conclude that $(X_n)_{n\ge 0}$ converges in probability to $\widetilde \gamma$ using a general argument, which we recall for completeness. For $n\ge 1$ and $\varepsilon>0$, let $\varphi_n(\varepsilon):=\P(X_n<\widetilde \gamma - \varepsilon)+\varepsilon$. 
For~$n$ large enough we have
\begin{align*}
	\widetilde{\gamma}+\varepsilon & \ge \E[X_n]  \\&\ge (\widetilde \gamma + \sqrt{\varphi_n(\varepsilon)})\cdot \P\left(X_n > \widetilde \gamma + \sqrt{\varphi_n( \varepsilon)}\right) + 
(\widetilde \gamma-\varepsilon)\cdot \P\left(\widetilde \gamma - \varepsilon \le X_n \le \widetilde \gamma + \sqrt{\varphi_n(\varepsilon)}\right)\\ 
&\ge \widetilde \gamma(1- \varphi_n(\varepsilon))  + \sqrt{\varphi_n(\varepsilon)}\cdot  \P\left(X_n > \widetilde \gamma+ \sqrt{\varphi_n( \varepsilon)}\right) - \varepsilon,
\end{align*}
which gives 
$$\P\left(X_n > \widetilde \gamma + \sqrt{\varphi_n(\varepsilon)}\right) \le  2\varepsilon(\varphi_n(\varepsilon))^{-1/2} + \widetilde{\gamma}(\varphi_n(\varepsilon))^{1/2}. $$
Together with \eqref{eq:main_bound_perc} this gives the convergence in probability of $X_n$ towards $\widetilde \gamma$.  

The proof of the theorem is almost finished now. Fix $\varepsilon>0$. On the one hand, the Markov inequality entails 
\begin{align*}
\P\left(\frac{1}{n^d} \log \uptau_{{\bf G}_n} \ge \widetilde \gamma + \varepsilon\right) & \le \P\left(X_n > \widetilde \gamma + \frac{\varepsilon}{2}\right) + 
e^{-(\widetilde \gamma+ \varepsilon)n^d}\cdot  \E\left[\uptau_{{\bf G}_n} \mathds 1\left\{\E[\uptau_{{\bf G}_n}\mid {\bf G}]\le e^{(\widetilde \gamma + \varepsilon/2)n^d}\right\} \right]\\
& \le \P\left(X_n > \widetilde \gamma + \frac{\varepsilon}{2}\right) + \exp\left\{-\frac{\varepsilon}2n^d\right\}.
\end{align*} 
On the other hand, Lemma~\ref{lem:inverse_cheb} gives
\begin{align*}
&\P\left(\frac{1}{n^d} \log \uptau_{{\bf G}_n} \le \widetilde \gamma - \varepsilon\right) \\[2mm]& \le \P\left(X_n \le \widetilde \gamma - \frac{\varepsilon}2\right) + e^{(\widetilde \gamma - \varepsilon)n^d}\cdot \E\left[\frac{1}{\E[\uptau_{{\bf G}_n}\mid {\bf G}]} \mathds 1\left\{\E[\uptau_{{\bf G}_n}\mid {\bf G}]\ge e^{(\widetilde \gamma - \varepsilon/2)n^d}\right\} \right]\\[2mm]
&  \le \P\left(X_n \le \widetilde \gamma - \frac{\varepsilon}2\right) + \exp\left\{-\frac{\varepsilon}2n^d\right\}. 
\end{align*}
Together with \eqref{conv.size.perco} this completes the proof of \eqref{conv.tau.perco} with $\gamma = \widetilde \gamma/\theta$. The proof of the theorem is now complete. 
\end{proof}

\subsection{Description of lattice models and verification of assumptions}
As mentioned earlier, we now give definitions of five models that fit in the above framework, and indicate references where proofs of all the hypotheses of the theorem are given for each of them.\\[2mm]
\textbf{Supercritical Bernoulli bond percolation, $d \geq 2$}. This is the basic percolation model defined by the prescription that each edge of $\Z^d$ is declared open with probability~$p$ and closed with probability~$1-p$, independently for all edges, where~$p \in [0,1]$. It is thoroughly studied in the standard reference of Grimmett~\cite{G}; there it is shown that there exists~$p_c(d) \in (0,1)$ such that there is no infinite component if~$p < p_c$ and there is a unique infinite component if~$p > p_c$. Here we assume that~$p > p_c$. Proofs of~\eqref{H0} and~\eqref{H0'} can be found in~\cite{G}. Condition~\eqref{H2} is proved in~\cite{Pis} in dimension~$d \geq 3$, and in~\cite{CM} for~$d= 2$. Condition~\eqref{H3} follows readily from independence.
 \\[2mm]
\textbf{Supercritical Bernoulli site percolation, $d \geq 3$}. This is defined as above, except that vertices, rather than edges, are declared to be open of closed. Site percolation is also covered in the book~\cite{G}. Again, a critical value~$p_c(d) \in (0,1)$ separates a non-percolative phase from a phase where there is  almost surely a unique infinite component. Here we assume that~$d \geq 3$ and $p > p_c(d)$. All the references and observations provided above for supercritical bond percolation are also valid for supercritical site percolation, except that~\eqref{H2} has not been established for the case~$d=2$, which is why we exclude it. \\[2mm]
\textbf{Occupied set of random interlacements, $d \geq 3$}. The model of random interlacements has been introduced by Sznitman in~\cite{Sz}, and detailed expositions are available in~\cite{interCT} and~\cite{interDRS}. It arises as a local limit of the trace of a simple random walk on a~$d$-dimensional torus ran for an amount of time proportional to the volume of the torus. The proportionality constant~$u > 0$ and the dimension~$d$ are the two parameters of the model; as the random walk is required to be transient, one takes~$d \geq 3$. An equivalent definition of the occupied set of random interlacements is the following: it is the unique random set of vertices~$\mathcal{I}^u \subset\Z^d$ satisfying
$$\P(\mathcal{I}^u \cap K = \varnothing) = \exp\{-u\cdot \text{cap}(K)\}$$
where~$K$ is any finite set of vertices of~$\Z^d$ and~$\text{cap}(K)$ denotes the discrete capacity of~$K$ (see Chapter~6 of \cite{lawler}). The density of~$\mathcal{I}^u$ is thus an increasing function of~$u$. Here we take any~$d \geq 3$ and any~$u > 0$. Our four conditions are proved in the following references. Conditions~\eqref{H0} and~\eqref{H0'} are given respectively by Theorem 2.1 and Corollary~2.3 in~\cite{Sz}. Condition~\eqref{H2} follows from Proposition~1 in~\cite{RSap}, and condition~\eqref{H3} is equation~(2.15) in~\cite{Sz} (one also needs the fact that the capacity of~$B_n$ is of order~$n^{d-2}$; see Section~6.5 in~\cite{lawler}).
\\[2mm]
\textbf{Vacant set of random interlacements in strongly percolative regime, $d \geq 3$}. For~$d \geq 3$ and~$u > 0$, the vacant set~$\mathcal{V}^u$ of random interlacements is defined as the complement of~$\mathcal{I}^u$. There exists~$u_* \in (0,\infty)$ such that~$\mathcal{V}^u$ has an infinite component if~$u < u_*$ and no infinite component if~$u > u_*$ (\cite{Sz}, \cite{SS}; see also \cite{rath} for a short proof). Conditions~\eqref{H0} and~\eqref{H3} follow from the same conditions for~$\mathcal{I}^u$, for which references were given above. For~$d \geq 3$ and~$u \in (0,u_*)$, condition~\eqref{H0'} is proved in~\cite{T}. Condition~\eqref{H2} is only known for~$u$ small enough (in the so-called strongly percolative regime of the vacant set), so we assume this regime is in force here. The proof is given in~\cite{T2} for~$d \geq 5$ and in~\cite{DRS} for all~$d\geq 3$; see  Theorem~2.5 in the latter reference. \\[2mm]
\textbf{Excursion set of Gaussian free field in strongly percolative regime, $d \geq 3$}. The Gaussian free field in~$\mathbb{Z}^d$,~$d \geq 3$, is the centered Gaussian field~$\varphi = (\varphi_x)_{x \in \mathbb{Z}^d}$ with covariances~$\mathbb{E}[\varphi_x\varphi_y] = g(x,y)$, where~$g$ denotes the Green function of simple random walk in~$\mathbb{Z}^d$. Its excursion sets are the sets $E^{\geq h}_\varphi = \{x \in \Z^d:\; \phi_x \geq h\}$ for~ $h \in \mathbb{R}$. There is again a critical value~$h_* \in \mathbb{R}$ so that~$E^{\geq h}_\varphi$ almost surely has an infinite component when~$h < h_*$ and only finite components if~$h > h_*$ (in fact it is also known that~$h_* > 0$). This has been proved in~\cite{BLM} for~$d =3$ and in~\cite{RS} for any~$d \geq 3$. Our condition~\eqref{H0} is proved in the paragraph preceding Lemma~1.5 in~\cite{RS}, and~\eqref{H0'} is verified in~\cite{RS}, Remark~1.6. Condition~\eqref{H2} is included in Theorem~2.7 in~\cite{DRS}; again this is only known in a strongly percolative regime where~$h$ is smaller than a constant~$\bar{h} \leq h_*$, so we assume this regime is in force here. Finally,~\eqref{H3} is given in Proposition~1.1 in~\cite{PR}.
\subsection{Random geometric graph}
The random geometric graph in $\R^d$, $d\ge 2$, is the random graph whose vertex set is a Poisson point process of intensity one, and the edge set is defined with the rule that two vertices 
are connected by an edge if and only if they are at Euclidean distance smaller than some fixed constant $R>0$. It has been shown that there exists a critical value $R_c>0$ such that the graph almost surely has only finite components if~$R < R_c$ and a uniue infinite component if $R>R_c$. This result and an in-depth treatment in the model can be found in the book of Penrose~\cite{Pen}. (We observe that the model is parametrized in a different manner there than the one we take here: the author fixes the radius for connectivity~$R = 1$ and takes the intensity of the Poisson process of vertices as the parameter of the model, which is denoted~$\lambda$ throughout the book. This choice is equivalent to ours after a rescaling of~$\mathbb{R}^d$).

As mentioned in the Introduction, we assume~$d \geq 2$, $R > R_c$, and let~${\bf G}_n$ be the connected component with highest cardinality in the graph restricted to~$B_n = [-n,n]^d$. Here we will prove that the statement of Theorem~\ref{thm:thm} holds for these graphs. The proof is very similar to that of Section~\ref{ss:general}, so we will only describe the points in which the proofs differ.

Below we state a proposition that contains all the properties that we will need concerning the graphs~${\bf G}_n$.  Before doing so, let us give two definitions. First, the metric diameter of a graph embedded in~$\mathbb{R}^d$ is defined as the supremum of the~$\ell_\infty$-distance between vertices of the graph. This is not to be confused with the graph-theoretic diameter. Second, consider the modified version of the random geometric graph on~$\mathbb{R}^d$ such that a vertex is artificially placed at the origin (and other vertices and edges are placed as before); let~$\theta_R$ be the probability that the vertex at the origin is in an infinite component. Then (see Chapter~9 of~\cite{Pen}), we have~$\theta_R > 0$ if and only if~$R > R_c$.
\begin{proposition}\label{prop:properties_RGG}
For any~$d \geq 2$,~$R > R_c$ and~$\varepsilon > 0$, there exists~$\delta > 0$ such that for~$n$ large enough the following holds with probability larger than~$1-\exp\{-(\log n)^{1+\delta}\}$. The maximal component~${\bf G}_n$ of the random geometric graph with parameter~$R$ on~$B_n$ has metric diameter larger than~$n$ and~${\displaystyle |{\bf G}_n|/(2n)^d \in ((1-\varepsilon) \cdot \theta_R, (1+\varepsilon) \cdot\theta_R)}$. Moreover, any other component has metric diameter smaller than~$(\log n)^2$ and cardinality smaller than~$\varepsilon n^d$.
\end{proposition}
\begin{proof}
This follows from putting together Proposition~10.13, Theorem~10.19, and Theorem~10.20 in~\cite{Pen}. 
\end{proof}
Note that Proposition~\ref{prop:properties_RGG} immediately gives
\begin{equation*}
\frac{|{\bf G}_n|}{(2n)^d} \xrightarrow[(\mathbb{P})]{n \to \infty} \theta_R.
\end{equation*}
We define~$X_n$ and~$\widetilde{\gamma}$ as in~\eqref{eq:def_of_X} and~\eqref{eq:def_of_gamma}. We would now like to show that~$\widetilde{\gamma} > 0$. However, unlike the lattice models considered earlier, here no universal upper bound is available for the degrees in~${\bf G}_n$. To remedy this, we use the following result.
\begin{lemma}
For any~$d \geq 2$ there exists~$K > 0$ such that~${\bf G}_n$ has a spanning tree with degrees bounded by~$K$.
\end{lemma}
The above statement for~$d=2$ (and~$K = 5$) is Lemma~2.12 in~\cite{tobias}. The proof easily generalizes to any dimension (with dimension-dependent~$K$). The idea is to take the spanning tree of~${\bf G}_n$ in which the sum of the lengths of the edges is the smallest possible, and to show that this tree has degrees bounded by a dimension-dependent constant. For this, one argues by contradiction: if the tree had a vertex whose degree was too large, then the edges in the ball of radius~$R$ around this vertex could be rearranged so as to produce a spanning tree with smaller total edge length.

With this result at hand, we prove that~$\widetilde{\gamma}$ is positive with the same computation as in~\eqref{eq:proof_of_positivity}, replacing~${\bf G}_n$ by its spanning tree with bounded degrees. From this point on, the proof of Theorem~\ref{thm:thm} for~$({\bf G}_n)$ proceeds in the same way as the one we gave for the lattice models; the only difference is that, in giving a lower bound to the probability of the event~$A$ defined in~\eqref{eq:def_of_A}, we use Proposition~\ref{prop:properties_RGG} instead of Proposition~\ref{prop:component} (also note that the computation involving covariances is unnecessary in this case, since graphs inside disjoint boxes are independent).

\section{Galton-Watson trees}\label{s:Galton-Watson}
In this section we consider a supercritical Galton-Watson tree with offspring distribution having a finite second moment. This means that if $\nu$ is the law of the number of individuals at the first generation of the tree, then $m:=\sum_{k\in \N} k\nu(k)$ is larger than one, and $\sigma^2:=\sum_{k\in \N} k^2\nu(k)$ is finite. 
We let~$Z_0 = 0$ and for $n\ge 1$, we let $Z_n$ be the number of individuals at generation (height) $n$, and let ${\bf G}_n$ be the subtree of individuals belonging to the first $n$ generations (including the root, which is in generation zero). In particular 
$|{\bf G}_n| = 1+ Z_1+\dots + Z_n$. We also denote the whole tree by ${\bf G}_\infty$. 
The events of non-extinction (or survival) and extinction are defined respectively by 
$$\text{Surv} = \left\{|{\bf G}_\infty| = \infty\right\},\quad\text{and}\quad \text{Ext} = \text{Surv}^c = \left\{|{\bf G}_\infty | < \infty\right\}.$$
Letting $v_n = 1 + m+ \cdots m^n$, it is well known that there exists a nonnegative random variable $W_\infty$ satisfying
\begin{equation}\label{eq:conv_l2}\frac{Z_n}{m^n} \to W_\infty,\quad \frac{|{\bf G}_n|}{v_n} \to W_\infty, \quad \text{almost surely and in $L^2$},\end{equation}
and moreover, \begin{equation}\label{eq:value_w}\mathbb{P}\left(W_\infty = 0\mid \text{Ext}\right) = 1,\quad\text{and}\quad \mathbb{P}\left(W_\infty > 0\mid \text{Surv}\right) = 1.\end{equation}
The convergence of~$Z_n/m^n$ (almost surely and in~$L^2$) and the two equalities in~\eqref{eq:value_w} are proved in Chapter~1 of~\cite{AN}. The almost sure convergence of~$|{\bf G}_n|/v_n$ is then immediate, and the convergence in~$L^2$ of~$|{\bf G}_n|/v_n$ can be obtained from Minkowski inequality:
\begin{align*}
\left(\mathbb{E}\left[\left(\frac{|{\bf G}_n|}{v_n} - W_\infty\right)^2\right]\right)^{1/2} \leq \sum_{i=0}^n \frac{m^i}{v_n}\left( \mathbb{E} \left[ \left( \frac{Z_i}{m^i} - W_\infty\right)^2\right] \right)^{1/2} \xrightarrow{n \to \infty} 0.
\end{align*}

We now define 
$$Y_n =\frac{\log \mathbb{E}[\uptau_{{\bf G}_n} \mid {\bf G}_n]}{m^n}, \quad\quad X_n = \frac{\log \mathbb{E}[\uptau_{{\bf G}_n} \mid {\bf G}_n]}{|{\bf G}_n|},$$
and  
$$\widetilde{\gamma} := \limsup_{n\to\infty} \mathbb{E}[Y_n].$$
We first show the following:
\begin{claim} ${\displaystyle \mathbb{E}[Y_n] \xrightarrow{n \to \infty} \widetilde{\gamma} > 0}$.
\end{claim}
\begin{proof}
On $\{Z_{1} = k\}$ for~$k > 0$, let ${\bf G}^{(1)},\ldots, {\bf G}^{(k)}$ denote the subtrees that descend from the vertices in the first generation. We have using Proposition \ref{supmult} that
$$\log\mathbb{E}[\uptau_{{\bf G}_n}\mid {\bf G}_n] \geq \sum_{i=1}^{Z_1} \log \mathbb{E}[\uptau_{{\bf G}^{(i)}} \mid {\bf G}_n] - CZ_{1}\log|{\bf G}_n|,$$
with the understanding that the right-hand side is zero when $Z_{1} = 0$. 
Hence,
\begin{align*}\mathbb{E}[Y_n] &\geq \frac{m^{n-1}}{m^n}\mathbb{E}\left[\sum_{i=1}^{Z_{1}} \frac{\log \mathbb{E}[\uptau_{{\bf G}^{(i)}} \mid {\bf G}_n]}{m^{n-1}} \right] - C\cdot  \mathbb{E}\left[\frac{Z_{1} \log|{\bf G}_n|}{v_{n}}\right]\\
&= \mathbb{E}[Y_{n-1}] - C\cdot  \mathbb{E}\left[\frac{Z_{1} \log|{\bf G}_n|}{m^n}\right].
\end{align*}
Using Cauchy-Schwarz inequality and the fact that~$\log^2(k) < \sqrt{k}$ for~$k$ large enough,
\begin{align*}
\mathbb{E}\left[Z_{1} \log|{\bf G}_n|\right]& \leq \left(\mathbb{E}[Z_1^2] \cdot \mathbb{E}[\log^2 |{\bf G}_n| ]\right)^{1/2}\\
&\leq \mathbb{E}[Z_1^2]^{1/2} \cdot \left( \log^2(m^{n^2}) + \mathbb{E}\left[\sqrt{|{\bf G}_n|} \cdot \mathds{1}\{|{\bf G}_n| \geq m^{n^2} \}\right] \right)^{1/2}\\
&\leq \mathbb{E}[Z_1^2]^{1/2} \cdot \left( n^4 \log^2(m) +  \left( \mathbb{E}[|{\bf G}_n|] \cdot \mathbb{P}\left(|{\bf G}_n| \geq m^{n^2} \right) \right)^{1/2}   \right)^{1/2}\\
&\leq \mathbb{E}[Z_1^2]^{1/2} \cdot \left( n^4 \log^2(m) +  \left( v_n^2/m^{n^2} \right)^{1/2}   \right)^{1/2} \leq n^3,
\end{align*}
if $n$ is large enough. We thus have, for $n$ large enough,
\begin{equation}\label{eq:first_bound_Y}
\mathbb{E}[Y_n] \geq \mathbb{E}[Y_{n-1}] - \frac{n^4}{m^n}.
\end{equation}

Next, using \eqref{eq:conv_l2} and \eqref{eq:value_w}, we can choose $\varepsilon > 0$ such that, for $n$ 
large enough,~$$\mathbb{P}(\varepsilon m^n  < |{\bf G}_n| < m^n/\varepsilon) > \varepsilon.$$ 
Then, if $n$ is large, using Theorem \ref{thm:lower}, 
\begin{equation} \label{eq:using_our_result}\begin{split}\mathbb{E}[Y_n]& \geq \mathbb{E}\left[\frac{\log\mathbb{E}[\uptau_{{\bf G}_n}\mid {\bf G}_n]}{m^n} \cdot \mathds{1}\{\varepsilon m^n < |{\bf G}_n| < m^n/\varepsilon \}\right]\\[.2cm]
&\geq \frac{1}{m^n} \mathbb{E}\left[\log \left( \exp \left(\frac{\csv|{\bf G}_n|}{\log^2|{\bf G}_n|} \right)\right) \cdot \mathds{1}\{\varepsilon m^n < |{\bf G}_n| < m^n/\varepsilon \}\right]    \geq \frac{c}{n^2},\end{split}
\end{equation}
for some constant $c>0$. 
Now, \eqref{eq:first_bound_Y} and \eqref{eq:using_our_result} can be combined to first give $\widetilde{\gamma} > 0$, and then again to give~$\mathbb{E}[Y_n] \xrightarrow{n \to \infty} \widetilde{\gamma}$.
\end{proof}

Define next $$\gamma = \frac{m-1}{m} \cdot \widetilde{\gamma}.$$ 
\begin{claim} \label{cl:first_Y} For all $\varepsilon > 0$, 
$$\mathbb{P}(X_n < \gamma - \varepsilon\mid Z_n \neq 0) \xrightarrow{n \to \infty} 0\quad  \text{ and } \quad \mathbb{P}(X_n < \gamma - \varepsilon\mid \textnormal{Surv}) \xrightarrow{n \to \infty} 0 .$$
\end{claim}
\begin{proof} Since~$\mathbb{P}(\{Z_n \neq 0\}\backslash \text{Surv})\xrightarrow{n \to \infty} 0$, any of the above convergences follows from the other, so we will prove the first one. 
For $n \in \mathbb{N}$, we write
$$n' = \lfloor n/2\rfloor,\qquad n'' = n - n'.$$
On $Z_{n'} = k > 0$, let ${\bf G}^{(1)},\ldots, {\bf G}^{(k)}$ denote the subtrees that descend from the vertices in generation $n'$. 
On $\{Z_n \neq 0\}$, we have
\begin{equation}\begin{split}
X_n&\ge \frac{1}{|{\bf G}_n|} \sum_{i=1}^{Z_{n'}} \log \mathbb{E}[\uptau_{{\bf G}^{(i)}}\mid {\bf G}_n] - \frac{cZ_{n'}\log |{\bf G}_n|}{|{\bf G}_n|}\\[.2cm]
&=\frac{v_n}{|{\bf G}_n|} \cdot \frac{m^{n}}{v_n} \cdot  \frac{Z_{n'}}{m^{n'}} \cdot \frac{1}{Z_{n'}} \sum_{i=1}^{Z_{n'}} \frac{\log \mathbb{E}[\uptau_{{\bf G}^{(i)}}\mid {\bf G}_n]}{m^{n''}} - \frac{cZ_{n'}\log |{\bf G}_n|}{|{\bf G}_n|}.
\end{split}\label{eq:ineq_x}\end{equation}
Conditionally on $\{Z_{n'} = k\}$ for~$k > 0$, the random variables~$(m^{n''})^{-1}\cdot \log \mathbb{E}[\uptau_{{\bf G}^{(i)}} \mid {\bf G}_n]$, for~$i = 1,\ldots, k$, are independent and all have the same distribution as $Y_{n''}$. Hence, for any $\varepsilon > 0$,
$$\mathbb{P}\left(\left. \left|\frac{1}{k}\sum_{i=1}^k \frac{\log \mathbb{E}[\uptau_{{\bf G}^{(i)}}\mid {\bf G}_n]}{m^{n''}} - \mathbb{E}[Y_{n''}]  \right|> \varepsilon\;\right|\; Z_{n'} = k  \right) \leq \frac{\mathbb{E}[(Y_{n''})^2]}{\varepsilon^2 k}.$$
Since for any $k > 0$ we have $\mathbb{P}(Z_{n'} \leq k \mid Z_{n'} \neq 0) \xrightarrow{n' \to \infty} 0$, we obtain
$$\mathbb{P}\left(\left. \left|\frac{1}{Z_{n'}}\sum_{i=1}^{Z_{n'}} \frac{\log \mathbb{E}[\uptau_{{\bf G}^{(i)}}\mid {\bf G}_n]}{m^{n''}} - \mathbb{E}[Y_{n''}]  \right|> \varepsilon\;\right|\; Z_{n'} \neq 0  \right) \xrightarrow{n' \to \infty} 0.$$
Further using the fact that $\mathbb{E}[Y_n] \xrightarrow{n \to \infty} \widetilde{\gamma}$ and $\mathbb{P}(Z_n = 0\mid Z_{n'} \neq 0) \xrightarrow{n\to\infty} 0$, the above gives
\begin{equation}\mathbb{P}\left(\left. \left|\frac{1}{Z_{n'}}\sum_{i=1}^{Z_{n'}} \frac{\log \mathbb{E}[\uptau_{{\bf G}^{(i)}}\mid {\bf G}_n]}{m^{n''}} - \widetilde{\gamma} \; \right|> \varepsilon\;\right|\; Z_{n} \neq 0  \right) \xrightarrow{n' \to \infty} 0.\label{eq:ineq_x2}\end{equation}
The desired result now follows from \eqref{eq:ineq_x}, \eqref{eq:ineq_x2} and the facts that
\begin{equation*}
\frac{m^n}{v_n}\xrightarrow{n \to \infty} \frac{m-1}{m},\qquad  \frac{|{\bf G}_n|}{v_n} \xrightarrow[\text{a.s.}]{n\to\infty} {W_\infty}, \qquad \frac{Z_{n'}}{m^{n'}} \xrightarrow[\text{a.s.}]{n \to \infty} W_\infty,
\end{equation*}
and
$$\lim_{n \to \infty} \mathbb{P}(W_\infty = 0\mid Z_n \neq 0) = 0.$$
\end{proof}

\begin{claim}\label{claim3} For all $\varepsilon > 0$, $$ \mathbb{P}(X_n > \gamma + \varepsilon\mid Z_n \neq 0) \xrightarrow{n \to \infty} 0 \quad \text{and}\quad \mathbb{P}(X_n > \gamma + \varepsilon\mid\textnormal{Surv}) \xrightarrow{n \to \infty} 0.$$
\end{claim}
\begin{proof} We will prove the second convergence. Assume by contradiction that there exists $\varepsilon > 0$ such that, for some increasing sequence $(n_k)$ with $n_k \to \infty$,
\begin{equation}\label{eq:for_contra}
\mathbb{P}(X_{n_k} > \gamma + \varepsilon \mid \text{Surv}) > \varepsilon.
\end{equation}
Fix a small $\delta > 0$ to be chosen later. Define the events
\begin{align*}
&B_{k,1} = \text{Surv}\cap \{ X_{n_k} < \gamma - \delta\},\\
&B_{k,2} = \text{Surv}\cap\{ \gamma - \delta \leq X_{n_k} \leq \gamma + \varepsilon\},\\
&B_{k,3} = \text{Surv}\cap \{ X_{n_k} > \gamma + \varepsilon\}.
\end{align*}
Also define
$$q_{k,0} = \mathbb{E}\left[\frac{|{\bf G}_{n_k}|}{v_{n_k}} \cdot \mathds{1}_\text{Ext}\right],\qquad q_{k,j} = \mathbb{E}\left[\frac{|{\bf G}_{n_k}|}{v_{n_k}} \cdot  \mathds{1}_{B_{k,j}}\right],\hspace{.3cm} j \in \{1,2,3\}.$$
Note that
\begin{equation}
\label{eq:sum_q} q_{k,0} + q_{k,1} +q_{k,2} +q_{k,3} =1.
\end{equation}
By \eqref{eq:conv_l2} and  \eqref{eq:value_w} we have
\begin{equation}\label{eq:bound_ext1}
q_{k,0} \leq \mathbb{E}\left[\left| \frac{|{\bf G}_{n_k}|}{v_{n_k}} - W_\infty\right|\right] + \mathbb{E}\left[W_\infty \cdot \mathds{1}_\text{Ext}\right] \xrightarrow{k \to \infty} 0.
\end{equation}
Similarly, we bound
\begin{equation}\label{eq:bound_ext2}
q_{k,1} \leq \mathbb{E}\left[\left| \frac{|{\bf G}_{n_k}|}{v_{n_k}} - W_\infty\right|\right] + \mathbb{E}\left[W_\infty \cdot \mathds{1}_{B_{k,1}}\right] \xrightarrow{k \to \infty} 0;
\end{equation}
the convergence follows from the fact that $\mathbb{P}(B_{k,1}) \xrightarrow{k \to \infty} 0$, by Claim \ref{cl:first_Y}. Next,
\begin{align*}
q_{k,3} \geq  \mathbb{E}\left[W_\infty\cdot \mathds{1}_{B_{k,3}} \right] -  \mathbb{E}\left[\left| \frac{|{\bf G}_{n_k}|}{v_{n_k}} - W_\infty\right|\right].
\end{align*}
Noting that \eqref{eq:for_contra} gives $\mathbb{P}(B_{k,3}) > \varepsilon$ for every $k$, we have
\begin{equation}\label{eq:bound_ext3}
\liminf_{k\to\infty} q_{k,3} \geq \varepsilon' := \inf_{B}\mathbb{E}[W_\infty\cdot \mathds{1}_B],
\end{equation}
where the infimum ranges over all sets $B$ with
$$B \in \sigma\left(\{Z_n:n \geq 0\}\right),\quad B \subseteq \text{Surv},\quad \mathbb{P}(B) > \varepsilon.$$
Since the law of $W_\infty$ conditioned on $\text{Surv}$ has no atom at zero, we obtain $\varepsilon' > 0$. 

We now put these estimates together. We start computing
\begin{align}
\nonumber\mathbb{E}[Y_{n_k}] &=\frac{v_{n_k}}{m^{n_k}}\cdot \left( \mathbb{E}\left[X_{n_k} \cdot \frac{|{\bf G}_{n_k}|}{v_{n_k}} \cdot \mathds{1}_\text{Ext} \right] + \sum_{j=1}^3 \mathbb{E}\left[ X_{n_k}\cdot \frac{|{\bf G}_{n_k}|}{v_{n_k}} \cdot \mathds{1}_{B_{k,j}}\right]\right)\\[.2cm]
\nonumber &\geq \frac{v_{n_k}}{m^{n_k}}\cdot\left[(\gamma - \delta)\cdot q_{k,2} + (\gamma + \varepsilon)\cdot q_{k,3}\right]\\[.2cm]
\label{eq:gamma_4} &\stackrel{\eqref{eq:sum_q}}{=} \frac{v_{n_k}}{m^{n_k}}\cdot\left[\gamma +q_{k,3}\cdot \varepsilon +q_{k,3}\cdot  \delta - \delta  + (q_{k,0} + q_{k,1}) (\delta  - \gamma)\right].
\end{align}
Using  \eqref{eq:bound_ext1}, \eqref{eq:bound_ext2}, \eqref{eq:bound_ext3} and the convergence $v_n/m^n \to m/(m-1)$, by  first choosing~$\delta$ small enough and then assuming $k$ is large enough, the expression in \eqref{eq:gamma_4} is larger than $\frac{m}{m-1}(\gamma + \varepsilon' \varepsilon/2)$. This gives
$$\liminf_{k \to \infty} \mathbb{E}[Y_{n_k}] > \widetilde{\gamma},$$
contradicting $\mathbb{E}[Y_n] \xrightarrow{n \to \infty} \widetilde{\gamma}$.
\end{proof}

The proof of Theorem~\ref{thm:thm} in the case of Galton-Watson trees now follows from Claims~\ref{cl:first_Y} and~\ref{claim3} by the same estimates as in the last paragraph of Section~\ref{ss:general}. Letting~$\bar{\mathbb{P}}(\cdot)$ denote either~${\mathbb{P}}(\cdot\mid\text{Surv})$ or~${\mathbb{P}}(\cdot\mid Z_n \neq 0)$ (and similarly for~$\bar{\mathbb{E}}$), we have
\begin{align*}
&\bar{\mathbb{P}}\left( \frac{1}{|{\bf G}_n|} \log \uptau_{{\bf G}_n} > \gamma + \varepsilon\right)\\[.2cm]&\leq \bar{\mathbb{P}}\left(X_n > \gamma + \frac{\varepsilon}{2}\right) + \bar{\mathbb{E}}\left[\mathbb{P}\left(\uptau_{{\bf G}_n} > e^{(\gamma + \varepsilon)|{\bf G}_n|}\mid {\bf G}_n\right)\cdot \mathds{1}{\left\{X_n \le \gamma + \frac{\varepsilon}{2} \right\}} \right]\\[.2cm]
&\leq \bar{\mathbb{P}}\left(X_n > \gamma + \frac{\varepsilon}{2}\right) + \bar{\mathbb{E}}\left[e^{-(\gamma + \varepsilon)|{\bf G}_n|}\cdot \mathbb{E}\left[\uptau_{{\bf G}_n}  \mid {\bf G}_n \right] \cdot \mathds{1}{\left\{  \mathbb{E}[\uptau_{{\bf G}_n} \mid {\bf G}_n] \leq e^{(\gamma + \frac{\varepsilon}{2})|{\bf G}_n|} \right\}} \right] \xrightarrow{n \to \infty} 0
\end{align*}
and
\begin{align*}
&\bar{\mathbb{P}}\left( \frac{1}{|{\bf G}_n|} \log \uptau_{{\bf G}_n} < \gamma - \varepsilon\right)\\[.2cm]&\leq \bar{\mathbb{P}}\left(X_n < \gamma - \frac{\varepsilon}{2}\right) + \bar{\mathbb{E}}\left[\mathbb{P}\left(\uptau_{{\bf G}_n} < e^{(\gamma - \varepsilon)|{\bf G}_n|}\mid {\bf G}_n\right)\cdot \mathds{1}{\left\{X_n \ge \gamma - \frac{\varepsilon}{2} \right\}} \right]\\[.2cm]
&\leq \bar{\mathbb{P}}\left(X_n < \gamma - \frac{\varepsilon}{2}\right) + \bar{\mathbb{E}}\left[e^{(\gamma - \varepsilon)|{\bf G}_n|}\cdot \mathbb{E}\left[\uptau_{{\bf G}_n}  \mid {\bf G}_n \right]^{-1} \cdot \mathds{1}{\left\{  \mathbb{E}[\uptau_{{\bf G}_n} \mid {\bf G}_n] \geq e^{(\gamma - \frac{\varepsilon}{2})|{\bf G}_n|} \right\}} \right] \xrightarrow{n \to \infty} 0.
\end{align*}
$\;$\\[-3mm]
\begin{center} {\textbf{Acknowledgements}} \end{center}$\;$\\[-8mm]
We would like to thank Bal\'azs R\'ath for directing us to several references on the models we study, leading to a wider applicability of our results. We would also like to thank Tobias M{\"u}ller for helpful discussions and references on the random geometric graph.

\end{document}